\documentclass[11pt]{amsart}

\usepackage{times} % Fuente de letras
\usepackage{enumerate}
\usepackage{esint}

\usepackage{amssymb}
\usepackage{cite}
\usepackage{pgf}
\usepackage{tikz}
\usetikzlibrary{patterns,shapes,calc,through, arrows,fadings,decorations.pathreplacing,intersections,snakes,automata}
\usepackage{pgfkeys}
\usepackage{verbatim}

\def\eps{\varepsilon}
\def\real{\mathbb{R}}
\def\supp{\mathop{\mbox{\normalfont supp}}\nolimits}
\def\sign{\mathop{\mbox{\normalfont sign}}\nolimits}

\def\loc{\mbox{\scriptsize{\normalfont{loc}}}}

\makeatletter\renewcommand\theenumi{\@roman\c@enumi}\makeatother

\newcommand\Mtilde{\stackrel{\sim }{\smash{M}\rule{0pt}{1.1ex}}}

\newtheorem{teo}{Theorem}%[section]

\newtheorem{lemma}[teo]{Lemma}
\newtheorem{prop}[teo]{Proposition}
\newtheorem{conjecture}{Conjecture}

\theoremstyle{definition}
\newtheorem*{xrem}{Remark}

\let\oldmarginpar\marginpar
\renewcommand\marginpar[1]{\-\oldmarginpar[\raggedleft\footnotesize #1]
{\raggedright\footnotesize #1}}

\begin{document}

\title[Muckenhoupt-Wheeden conjectures in higher dimensions.]{Muckenhoupt-Wheeden conjectures in higher dimensions.}

\author[A. Criado]{Alberto Criado}
\address{Alberto Criado \newline\indent Departamento de Matem\'aticas, Universidad Aut\'onoma de Madrid, 28049 Madrid, Spain}
\email{alberto.criado@uam.es}
\thanks{Authors supported by DGU grant MTM2010-16518.}

\author[F. Soria]{Fernando Soria}
\address{Fernando Soria \newline\indent Departamento de Matem\'aticas and Instituto de Ciencias Matem\'aticas CSIC--UAM--UC3M--UCM, Universidad Aut\'onoma de Madrid, 28049 Madrid, Spain}
\email{fernando.soria@uam.es}

\begin{abstract}

In recent work by Reguera and Thiele \cite{RegueraThiele} and by Reguera and Scurry \cite{RegueraScurry}, two conjectures about joint weighted estimates for  Calder\'on-Zygmund operators and the Hardy-Littlewood maximal function have been refuted in the one-dimensional case. One of the key ingredients for these results is the construction of weights for which the action of the Hilbert transform is substantially bigger than that of the maximal function. In this work, we show that a similar construction is possible for classical Calder\'on-Zygmund operators in higher dimensions. This allows us to fully disprove the conjectures.

\end{abstract}


\maketitle

\section{Introduction and statements of results}

In this paper we will study joint weighted estimates for the Hardy-Littlewood maximal operator and classical Calder\'on-Zygmund operators. We consider the non-centered Hardy-Littlewood maximal operator over cubes, defined for a locally integrable function $f$ as
\[
Mf(x)=\sup_{x\in Q\subset \mathcal Q} \fint_Q |f(y)|\,dy,
\]
where $\mathcal Q$ denotes the family of all cubes with sides parallel to the coordinate axes in $\real^d$. We will also consider classical Calder\'on-Zygmund singular integral operators, whose action on a smooth function $f$ is defined by
\[
Tf(x)=p.v.\int_{\real^d} K(x,y)f(y)\,dy.
\]
Here the kernel $K$ has the form
\begin{equation}\label{nucleo}
K(x,y)=\frac{\Omega(x-y)}{|x-y|^d},
\end{equation}
with $\Omega$ a homogeneous function of degree 0, such that $\Omega\in C^1(\mathbb S^{d-1})$ and $\int_{\mathbb S^{d-1}}\Omega(x)\,d\sigma_{d-1}(x)=0$. The Hilbert transform in one dimension and the Riesz transforms in higher dimensions are examples of such operators. We may also consider more general Carder\'on-Zygmund operators. In fact, our arguments work well for operators with variable kernels $K$ satisfying standard size and regularity conditions. We will not pursue here these generalizations. Instead, we will make some comments on how to extend our results to this more general setting.

\bigskip

In this context, a weight simply means a non-negative function $w:\real^n\rightarrow [0,\infty]$. Such $w$ can be interpreted as the density of an absolutely continuous measure. This measure is usually denoted by the same letter as its density. That is, if $w$ is a weight in $\real^n$, for a measurable $E\subset \real^n$ one writes $w(E)=\int_E w(x)\,dx$ and for $f$ a measurable function we say that $f\in L^p(w)$ if $\|f\|_{L^p(w)}=\big(\int |f|^p w\bigr)^{1/p}<\infty$.

\bigskip

In the 1970's B. Muckenhoupt and Wheeden among other authors began the study of weighted inequalities for maximal, Calder\'on-Zygmund and other operators. They defined the $A_p$ class as the collection of weights $w$ satisfying 
\begin{equation}\label{ap.un.peso}
\sup_{Q\in \mathcal Q} \fint_Q w(y)\,dy \left(\fint_Q w(y)^{-p'/p}\,dy\right)^{p/p'} < \infty,
\end{equation}
if $1<p<\infty$, or 
\begin{equation}\label{a1.un.peso}
Mw(x) \leq C w(x)\ \ \mbox{a.e. } x,
\end{equation}
with $C>0$ independent of $x$, if $p=1$. It is well known that $w\in A_p$ is equivalent to $M$ being bounded on $L^p(w)$, if $p>1$, and to $M$ being weakly bounded on $L^1(w)$, if $p=1$. It is also known that (\ref{ap.un.peso}) and (\ref{a1.un.peso}) are sufficient too for the same kind of estimates of a Calder\'on-Zygmund operator, but only necessary in the sense that if all the $d$ Riesz transforms are weakly bounded on $L^p(w)$, then $w\in A_p(w)$, for $1\leq p<\infty$. In the one dimensional case this means in particular that the Hilbert transform is weakly bounded on $L^p(w)$ if and only if $w\in A_p$, for $1\leq p<\infty$. For a more complete account on these facts see \cite{GarciaCuervaRubiodeFrancia} and \cite{Grafakos}.

\bigskip

The situation is more complicated when one considers norm estimates with two weights. A pair of weights $(u,v)$ is in the $A_p$ class if
\begin{equation}\label{ap.dos.pesos}
\sup_{Q\in \mathcal Q} \fint_Q v(y)\,dy \left(\fint_Q u(y)^{1-p'}\,dy\right)^{1/p'}<\infty,
\end{equation}
for $p>1$, and in $A_1$ if
\begin{equation}\label{a1.dos.pesos}
Mv(x)\leq C u(x),\ \ \mbox{a.e. } x,
\end{equation}
with $C>0$ independent of $x$. These conditions are equivalent to the mapping $M:L^p(u)\rightarrow L^{p,\infty}(v)$ to be bounded, for $1\leq p<\infty$, and necessary for the strong boundedness $M:L^p(u)\rightarrow L^p(v)$, if $p>1$, but not sufficient for it. The continuity of $M$ from $L^p(u)$ to $L^p(v)$ was, nevertheless, characterized by E. Sawyer \cite{Sawyer} to be equivalent to
\[
\int_Q M(\chi_Q v^{1-p'})^p u \leq C \int_Q v^{1-p'}<\infty,
\]
for all $Q\in \mathcal Q$. In the one weight setting some of the norm estimates for Calder\'on-Zygmund operators were shown to be equivalent to the ones for $M$. This suggested that similar connections might be found in the two weight setting. B. Muckenhoupt, R. Wheeden and others proposed several of them. For many years they could not be confirmed or refuted and became known as Muckenhoupt-Wheeden conjectures. 

\bigskip

Perhaps the most famous one originates in a result by C. Fefferman and E.M. Stein \cite{FeffermanStein} showing that there is an absolute constant such that for any weight $w$ one has
\begin{equation}\label{desigualdad.feffermanstein}
w\left(\{x\in\real^d: Mf(x)>\lambda\}\right) \leq \frac C\lambda\int f(x) Mw(x)\,dx.
\end{equation}
It has been attributed to Muckenhoupt and Wheeden the conjecture that the same two weight inequality should be true for a Calder\'on-Zygmund operator.

\bigskip

\begin{conjecture}\label{conj1} For each classical Calder\'on-Zygmund operator $T$,  there exists a constant $C>0$ so that for every weight $w$ one has
\begin{equation}\label{conjetura}
w\left(\{x\in\real^d: |Tf(x)|>\lambda\}\right) \leq \frac C\lambda\int |f(x)| Mw(x)\,dx,
\end{equation}
for all $\lambda>0$ and $f\in L^1(Mw)$.
\end{conjecture}

\bigskip

The question was extended to more general operators and the conjecture was shown to be true for some square functions in \cite{ChanilloWheeden}, but false for fractional integral operators in \cite{CarroPerezSoriaSoria}. The closest approach, on the positive side, for Calder\'on-Zygmund operators is due to C. P\'erez, who showed in \cite{Perez1} that (\ref{conjetura}) is true if $M$ is replaced by the iterated operator $M^2$ or even by the operator $M_{L(\log L)^\eps}$, with $\eps>0$. Later, C. P\'erez and D. Cruz-Uribe \cite{CUCP2}, used the extrapolation technique to show that if (\ref{conjetura}) holds for a sublinear operator $T$, then one has
\begin{equation}\label{cond.perez.cu}
\int |Tf(x)|^p w(x)\,dx \leq C \int |f(x)|^p\left(\frac{Mw(x)}{w(x)}\right)^p w(x)\,dx,
\end{equation}
for all $p>1$. This necessary condition was disproved by M.C. Reguera and C. Thiele in \cite{RegueraThiele} in the case $p=2$, thus showing the conjecture to be false. They gave a counterexample in the one-dimensional case, that is, when $T$ is the Hilbert transform. The construction was based on a simplification of the technique used by M.C. Reguera in \cite{Reguera} in order to refute the corresponding assertion in the dyadic setting. 

\bigskip

Our first result shows that Conjecture \ref{conj1} is false for all classical Calder\'on-Zygmund operators.

\bigskip

\begin{teo}\label{no.conj1} Let $T$ be a Calder\'on-Zygmund operator with an associated kernel satisfying (\ref{nucleo}). Then, $\forall N>0$, $\exists w$ weight, $\exists f\in L^1(Mw)$ and $\exists \lambda>0$ so that
\begin{equation}\label{desigualdad.teorema}
w(\{|Tf|>\lambda\})\geq \frac N\lambda\int |f| Mw.
\end{equation}
\end{teo}

\bigskip

D. Cruz-Uribe, C. P\'erez and J.M. Martell in \cite{CUMP3} considered another conjecture relating two weight estimates for the maximal operator and Calder\'on-Zygmund operators. This conjecture is also attributed to Muckenhoupt and Wheeden and its precise statement is the following.

\bigskip

\begin{conjecture}\label{conj2}
Let $T$ be a Calder\'on-Zygmund operator as above, then
\[
\left.\begin{array}{c} M:L^p(u)\rightarrow L^p(v)\\ M:L^{p'}(v^{1-p'})\rightarrow L^{p'}(u^{1-p'})\end{array}\right\}\quad \Longrightarrow \quad T:L^p(u)\rightarrow L^p(v).
\]
\end{conjecture}

\bigskip

\begin{xrem}
To simplify the notation throughout this work, the symbol \lq 
$S:X\rightarrow Y$' will always mean that the operator $S$ maps the elements of the space $X$ into elements of $Y$ in a continuous way. This notation has been already used  in the statement of the above conjecture.
\end{xrem}

\bigskip

The motivation for the second condition on $M$ is the following. A simple duality argument shows that since $T$ is an essentially self-adjoint operator, $T:L^p(u)\rightarrow L^p(v)$ is equivalent to $T:L^{p'}(v^{1-p'})\rightarrow L^{p'}(u^{1-p'})$. 

\bigskip

This conjecture was refuted by M.C. Reguera and J. Scurry in \cite{RegueraScurry} for the Hilbert transform. Their counterexample is based on the one that disproved Conjecture \ref{conj1} in \cite{RegueraThiele}. We show that the conjecture is false again for every classical Calder\'on-Zygmund operator.

\bigskip

\begin{teo}\label{no.conj2}
Fix $1<p<\infty$, and let $T$ be a Calder\'on-Zygmund operator as in Theorem \ref{no.conj1}. Then one can construct weights $u$ and $v$ such that $M:L^p(u)\rightarrow L^p(v)$ and $M:L^{p'}(u^{1-p'})\rightarrow L^{p'}(v^{1-p'})$ but there exists an $f\in L^p(u)$ such that $\|Tf\|_{L^p(v)}=\infty$.
\end{teo}

\bigskip

One important observation is that while an $A_p$ weight is a.e. positive, the previous results have no assumptions on the support of the weight. In order for the questions we are treating to make sense, for $w$ a weight vanishing in some set of positive Lebesgue measure, we define $L^p(w)$ as the space of the measurable functions $f$ so that $\supp f\subset \supp w$ and $\|f\|_{L^p(w)}<\infty$\footnote {In a similar fashion, the expression $w^\alpha(x)$ for negative $\alpha$ is set to be zero at the points $x$ where $w(x)=0$.}. Indeed, one of the key ingredients in the proofs in \cite{RegueraThiele} and \cite{RegueraScurry} is to consider weights with sparse support. In \cite{RegueraScurry} it is shown that  in the one-dimensional setting these weights do not preserve the equivalence of the boundedness of $M$ and $H$ on weighted $L^p$. We will extend this, showing that unlike for a.e. positive weights, in this setting the boundedness of $M$ on $L^p(w)$ does not imply the same result for  Calder\'on-Zygmund operators.

\bigskip

\begin{teo}\label{no.conj3} Let $T$ be a Calder\'on-Zygmund operator. Then there exist a weight $u$ and a function $f\in L^p(u)$ such that $M$ is bounded on $L^p(u)$ but $\|Tf\|_{L^p(u)}=\infty$.
\end{teo}

\bigskip

Although our work does not make any contribution to them, for completeness we briefly comment still other important Muckenhoupt-Wheeden conjectures. Conjecture \ref{conj2} had a weak version asserting that $M:L^{p'}(u^{1-p'})\rightarrow L^{p'}(v^{1-p'})$ implies $T:L^p(u)\rightarrow L^{p,\infty}(v)$, for $T$ a Calder\'on-Zygmund operator. This has been shown to be false for the Hilbert transform by D. Cruz-Uribe, A. Reznikov and A. Volberg in \cite{CURV}. By duality, Conjecture \ref{conj1} implied this last conjecture. Thus, the argument in \cite{CURV} also refutes the one-dimensional case of Conjecture \ref{conj1} in an indirect way.

\bigskip

At last, we mention a still open conjecture. It asserts that replacing the $L^p$ or $L^{1-p'}$ integrability requirement in (\ref{ap.dos.pesos}) by a slightly stronger one in the sense of Orlicz integrals will be enough to guarantee the $L^p$ boundedness of Calder\'on-Zygmund operators. This is known as the bump conjecture and only partial results have been obtained so far. For more details see \cite{CUMP1, CUMP2, CUMP3, CUCP1, CUCP2, CUCP3, CUCP4, CURV, Lerner, NazarovReznikovVolberg, Perez0, Perez1, Perez3, PerezWheeden, TreilVolbergZheng}. 

\bigskip

The rest of the paper is organized as follows. In Section 2 we prove Theorems \ref{no.conj1}, \ref{no.conj2} and \ref{no.conj3} assuming the existence of some weights satisfying certain specific properties. Section 3 is devoted to the construction of these weights. As usual, $C$ and $c$ will denote positive constants, that may have different values at different occurrences. Also, given two quantities $A,B>0$, by $A\sim B$ we mean that there exist a constant $C>0$, which may depend on the dimension but is independent otherwise of the main parameters involved, such that $A\leq CB$ and $B\leq CA$.

\bigskip

\section{Proofs of the Theorems.}

\bigskip

The proofs of the three Theorems stated in the previous section are based on the construction of weights satisfying a local $A_1$ property but allowing large values under the action of a given Calder\'on-Zygmund operator.

\bigskip

\begin{prop}\label{prop.pesos} Let $T$ be a Calder\'on-Zygmund operator with an associated kernel satisfying (\ref{nucleo}). Then, for each sufficiently large $N\in \mathbb N$, there exists a weight $w_N$ so that if we denote $D_N:=\supp w_N\subset [0,1]^d$ we have both, $w_N\geq 1$ and  $Mw_N\leq C w_N$ on $D_N$ and $|T w_N|\geq CN w_N$ on $\widehat{D_N} \subset D_N$, with $|\widehat{D_N}|\sim |D_N|$ and $w_N(\widehat{D_N})\sim w_N(D_N)=1$.
\end{prop}

\bigskip

The conclusion $Mw_N\leq C w_N $ in the support of $ w_N$ is what makes  $ w_N$ an $A_1$ weight in a local sense. We will first prove Theorems \ref{no.conj1}, \ref{no.conj2} and \ref{no.conj3} assuming that Proposition \ref{prop.pesos} is true, leaving its proof for the next section.

\bigskip

\begin{proof}[Proof of Theorem \ref{no.conj1}] Consider $T^\ast$ the adjoint operator of $T$. Note that $T$ is an essentially self-adjoint operator, indeed we have $T^\ast f(x)= Tf(-x)$. Given $N>0$ consider the weight $w_N$ associated to $T^\ast$ from Proposition \ref{prop.pesos}. Taking $f=w_N T^\ast w_N/(Mw_N)^2$, we have
\begin{equation}\label{positiva}
\int Tf\ w_N\ =\ \int f\ T^\ast w_N\ =\ \int \left|\frac{T^\ast w_N}{Mw_N}\right|^2 w_N\ \geq CN^2\,w_N(\widehat{D_N})\geq CN^2> 0.
\end{equation}
Considering $F$ to be the non-increasing rearrangement of $|Tf|$ with respect to $w_N$ in $\real^d$, we also have
\begin{eqnarray}\label{lambda}
\int \left|\frac{T^\ast w_N}{Mw_N}\right|^2 w_N&=&\int Tf\ w_N \leq \int |Tf|\ w_N\ = \int_0^{w_N(\real^d)} F(t)\,dt \leq \int_0^1 \frac{dt}{t^{1/2}}\  \sup_{s>0} s^{1/2} F(s) \nonumber\\&=& 2\ \sup_{\lambda>0} \lambda\ w_N(\{|Tf|>\lambda\})^{1/2} \nonumber\\&\leq& 3\ \lambda_0\ w_N(\{|Tf|>\lambda_0\})^{1/2},
\end{eqnarray}
for some $\lambda_0$. Combined with (\ref{positiva}), this yields 
\begin{equation}\label{idea}
\left(\int \left|\frac{T^\ast w_N}{Mw_N}\right|^2 w_N\right)^{1/2}\leq \frac CN \int \left|\frac{T^\ast w_N}{Mw_N}\right|^2 w_N \leq \frac CN\,\lambda_0\, w_N(\{|Tf|>\lambda_0\})^{1/2}.
\end{equation} 
Now we define $E=\{|Tf|>\lambda_0\}$ and $w=\chi_Ew_N$. Using H\"older's inequality and (\ref{idea}) we have
\[
\int |f|\ Mw = \int \frac{w_N\,T^\ast w_N}{(Mw_N)^2} Mw \leq \left(\int \Bigl|\frac{T^\ast w_N}{Mw_N}\Bigl|^2w_N\right)^{1/2}\left(\int \Bigl|\frac{Mw}{Mw_N}\Bigr|^2w_N\right)^{1/2} \leq \frac CN\ \lambda_0\ w_N(E),
\]
the last inequality provided we show the following Lemma.

\bigskip

\begin{lemma}
There exists a constant $C>0$ so that for all weights $v$ and all measurable sets $E\subset \real^n$ one has
\begin{equation}\label{desigualdad}
\left(\int \Bigl|\frac{M(\chi_Ev)}{Mv}\Bigr|^2 v\right)^{1/2} \leq C v(E)^{1/2}.
\end{equation}
\end{lemma}

\bigskip

\noindent {\it Proof}.
Given a weight $v$ we define the operator $S_{v}$ for $f\in L_{\loc}^1(v)$ as
\[
S_{v}f(x)=\frac{M(fv)(x)}{Mv(x)}.
\]
We will prove indeed a stronger result, that for all $p>1$ one has 
\[
\int |S_vf|^p v\leq C \int |f|^pv.
\]
Since $M(fv)\leq \|f\|_{L^\infty(v)} Mv$, one has that $S_v$ is bounded on $L^\infty(v)$ with operator norm $1$. By interpolation, the result is proved if we show that $S_{v}$ is of weak type $L^1(v)$ with a constant independent of $v$. Since it makes no essential difference, we will see it for $\widetilde S_vf= \Mtilde(fv)/\Mtilde v$, where $\Mtilde$ denotes the centered maximal operator. Let $f\in L^1(v)$ and $0<\lambda<1$. If $\widetilde S_{v}f(x)>\lambda$, there exists $R_x>0$ so that 
\[
\fint_{Q(x,R_x)} |f|v  >\lambda\,\Mtilde v(x) \geq \lambda\, \fint_{Q(x,R_x)} v>0,
\]
where by $Q(x,R)$ we mean the cube in $\mathcal Q$ of edge length $R$ and centered at $x$. This implies that 
$$v(Q(x,R_x))\leq \frac1\lambda \int_{Q(x,R_x)} |f|v.$$ 
Observe that the cubes $Q(x,R_x)$ with $x\in A_\lambda:=\{x\in\real^d: \widetilde Sf(x)>\lambda\}$ are a Besicovitch cover of $A_\lambda$. By Besicovitch Covering Theorem (see \cite{Guzman}) there is a subcover by cubes $Q(x,R_x)$, with $x\in A_\star\subset A_\lambda$, such that each $x\in \real^d$ belongs to at most  $b_d$ cubes of the subcover, where $b_d$ is a number that only depends on the dimension. Then we have
\[
v(A_\lambda)\leq v\Bigl(\bigcup_{x\in A_\star}Q(x,R_x)\Bigr) \leq \sum_{x\in A_\star} v(Q(x,R_x)) \leq \frac1\lambda \sum_{x\in A_\star} \int_{Q(x,R_x)} |f|v\leq \frac{b_d}\lambda \int |f|v.
\]
This proves the lemma and, hence, Theorem \ref{no.conj1} too.
\end{proof}

\bigskip

Let us now prove Theorem \ref{no.conj3}.

\bigskip

\begin{proof}[Proof of Theorem \ref{no.conj3}]
We use the same `hump gliding' argument as in \cite{RegueraScurry}. Let $z\in \real^d$ be a unitary vector. We define $w:=\sum_{N=N_0}^\infty \widetilde w_N$, where $\widetilde w_N(x)=w_N(x-3^Nz)$ and $w_N$ are the weights described in Proposition \ref{prop.pesos}, starting at some $N_0$ large. We also define  $g:=\sum_{N=N_0}^\infty \frac1{N^\eps}\chi_{Q_N}$ with $Q_N=[0,1]^d+3^Nz$ and $1/p<\eps<1$. Finally, we take $u=w^{1-p}$ and $f=gw$.

\bigskip

First, we check that  $f\in L^p(u)$:
\begin{equation*}
\int |f|^p u =  \int g^p w= \sum_{N=N_0}^\infty \int_{Q_N} \frac1{N^{\eps p}}w_N(x-3^Nz)\,dx = \sum_{N=N_0}^\infty \frac1{N^{\eps p}}<\infty.
\end{equation*}

\bigskip

Next, we see that $Tf\not\in L^p(u)$. In order to do so, we write $\|Tf\|_{L^p(u)}$ as
\[
\left(\sum_{N=N_0}^\infty \int_{Q_N}\Bigl|\frac{1}{N^\eps} T\widetilde w_N(x)+\sum_{J\neq N} \frac1{J^\eps} T\widetilde w_J(x)\Bigr|^p \widetilde w_N(x)^{1-p}\,dx\right)^{1/p}.
\]
By the triangle inequality this is greater than or equal to $A-B$, where
\begin{eqnarray*}
A&=&\left(\sum_{N=N_0}^\infty \int_{Q_N}\Bigl|\frac{1}{N^\eps} T\widetilde w_N(x)\Bigr|^p \widetilde w_N(x)^{1-p}\,dx\right)^{1/p},\\ B&=&\left(\sum_{N=N_0}^\infty \int_{Q_N}\Bigl|\sum_{J\neq N} \frac1{J^\eps} T\widetilde w_J(x)\Bigr|^p \widetilde w_N(x)^{1-p}\,dx\right)^{1/p}.
\end{eqnarray*}
We will see that $A=\infty$ and $B<\infty$. We begin with $B$. If $x\in Q_N$ and $J\neq N$ we have
\begin{eqnarray*}
|T\widetilde w_J(x)|&\leq& \int_{Q_J} |K(x-y)|w_J(y-3^Jz)\,dy \leq \int_{R_J} \frac{C}{|3^N-3^J|^d} w_J(y-3^Jz)\,dy \\&\leq& \frac{C}{\max\{3^N,3^J\}^{d}} w_J([0,1]^d)\leq  \frac C{3^{dN/2}3^{dJ/2}}.
\end{eqnarray*}
Here we have used that for $y\in Q_J$ and $J\neq N$ one has $|x-y|\sim |3^N-3^J|\sim 3^N+3^J$. Hence,
\begin{eqnarray*}
B^p &\leq& C\sum_{N=N_0}^\infty \int_{Q_N}\Bigl|\sum_{J\neq N} \frac1{J^\eps}\ \frac1{3^{dN/2}3^{dJ/2}} \Bigr|^p w_N(x-3^Nz)^{1-p}\,dx
\\ &\leq& C\sum_{N=N_0}^\infty \Bigl|\sum_{J\neq N} \frac1{J^\eps}\ \frac1{3^{dN/2}3^{dJ/2}} \Bigr|^p < \infty.
\end{eqnarray*}
Now we proceed with $A$. Using an obvious change of variables in the integration and the property that $|Tw_N|\geq CN w_N$ in $\widehat{D_N}$ we have
\begin{eqnarray*}
A^p &=& \sum_{N=N_0}^\infty \frac1{N^{\eps p}} \int_{[0,1]^d}|Tw_N(x)|^p w_N(x)^{1-p}\,dx \geq \sum_{N=N_0}^\infty \frac1{N^{\eps p}} \int_{\widehat{D_N}}|Tw_N(x)|^p w_N(x)^{1-p}\,dx\\ &\geq& C \sum_{N=N_0}^\infty \frac {N^p}{N^{\eps p}} \int_{\widehat{D_N}} w_N(x)\,dx \geq C\sum_{N=N_0} N^{p(1-\eps)}=\infty.
\end{eqnarray*}

\bigskip

It remains to prove that $M$ is bounded on $L^p(u)$. Since it makes no essential difference we will prove it for the centered maximal operator $\Mtilde$ again. We define $\mathcal Q_w=\{Q\in \mathcal Q: w(Q)>0\}$. For $f\in L^p(u)$ and $Q\in \mathcal Q_w$ we have
\[
\frac1{|Q|}\int |f| = \frac{w(Q)}{|Q|}\ \frac1{w(Q)}\int_Q |fw^{-1}|w.
\]
This implies that
\[
\Mtilde f\leq \Mtilde\!w\ \Mtilde_w(fw^{-1}),
\]
where $\Mtilde_w$ is the centered maximal operator associated to $w$ defined by
\[
\Mtilde_w g(x)=\sup_{R>0,w(Q(x,R))>0} \frac1{w(Q(x,R))}\int_{Q(x,R)} |g|\,w.
\]
It is easy to check that for $x\in Q_N$ one has $Mw(x)\sim Mw_N(x-3^Nz)\leq C w_N(x-3^Nz)=Cw(x)$, that is 
\begin{equation}\label{a1.locala}
Mw\sim w\quad \mbox{ in } \supp w.
\end{equation}
Hence, since the same is true for $\Mtilde$, we have
\[
\int |\Mtilde f|^p w^{1-p} \leq \int |\Mtilde w|^p |\Mtilde_w(fw^{-1})|^p w^{1-p}
\leq C \int |\Mtilde_w(fw^{-1})|^p w.
\]
A well-known consequence of Besicovitch Covering Theorem is that $\Mtilde_w$ is bounded on $L^p(w)$. This, together with the observation that $f\in L^p(w^{1-p})$ if and only if $fw^{-1}\in L^p(w)$, finishes the proof.

\end{proof}

\bigskip

We now present the proof of Theorem \ref{no.conj2}. As we will see, everything  reduces to the same arguments used in the proof of Theorem \ref{no.conj3}.

\bigskip

\begin{proof}[Proof of Theorem \ref{no.conj2}]
At this point we assume that the reader is familiar with the notation and the circle of ideas surrounding the proof of Theorem \ref{no.conj3}. Taking again $w(x)=\sum_{N=N_0}^\infty w_N(x-3^Nz)$ we consider the weights $u=(Mw/w)^pw$ and $w$. In view of (\ref{a1.locala}), we have $u\sim w$ in $W=\supp w$, which  reduces the problem to the one weight setting.

\bigskip

It is easy to see that for an essentially self-adjoint operator $T$, the following inequalities are equivalent
\begin{eqnarray}\label{acotacion.normal}
\|Tf\|_{L^p(w)}&\leq& C_\star \|f\|_{L^p(u)},\label{uno}\\
\|T(fu^{1-p'})\|_{L^p(w)} &\leq& C_\star \|f\|_{L^p(u^{1-p'})},\nonumber\\
\|T(fw)\|_{L^{p'}(u^{1-p'})}&\leq& C_\star \|f\|_{L^{p'}(w)}.\label{tres}
\end{eqnarray}

\bigskip

Instead of (\ref{uno}) we will disprove (\ref{tres}). Taking again $g=\sum_{N=N_0}^\infty\frac1{N^\eps} \chi_{Q_N}$, with $1/p<\eps<1$, we have that $g\in L^{p'}(w)$. On the other hand,  
\[
\|T(gw)\|_{L^{p'}(u^{1-p'})}^{p'}=\int |T(gw)|^{p'} \frac{w}{(Mw)^{p'}} \geq C\int |T(gw)|^{p'} w^{1-p'},
\]
and this last quantity was shown to be infinite in the proof of Theorem \ref{no.conj3}, except that the roles of $p$ and $p'$ were interchanged.

\bigskip

To prove $M:L^p(u)\rightarrow L^p(v)$ is easy. For $f\in L^p(u)$, using Fefferman-Stein inequality (\ref{desigualdad.feffermanstein}) and (\ref{a1.locala}), we have
\[
\|Mf\|_{L^p(v)}^p=\int |Mf|^p w \leq C \int |f|^p Mw \leq C \int |f|^pw \leq C\int |f|^p \Bigl(\frac{Mw}{w}\Bigl)^p w = C\|f\|_{L^p(v)}^p.
\]
We finish showing that $M:L^{p'}(v^{1-p'})\rightarrow L^{p'}(u^{1-p'})$. Similarly as before, for $f\in L^{p'}(w^{1-p'})$ we have
\[
\|Mf\|_{L^{p'}(u^{1-p'})}^{p'}=\int |Mf|^{p'} \frac{w}{(Mw)^{p'}} \leq \int |Mf|^{p'} w^{1-p'} \leq C\int |f|^{p'} w^{1-p'},
\]
where the last inequality was obtained in the proof of Theorem \ref{no.conj3} for $p$ instead of $p'$.
\end{proof}

\bigskip

\section{The construction of the weights}

\bigskip

The construction of the weights $w_N$ in Proposition \ref{prop.pesos} is an extension to higher dimension of the one by M.C. Reguera and C. Thiele in \cite{RegueraThiele}, which in turn was a simplification of the construction by M.C. Reguera in \cite{Reguera}. The argument is long and involves some technicalities.

\bigskip

\begin{proof}[Proof of Proposition \ref{prop.pesos}.]  First we will give the basics of the construction of the weight $w_N$ and of the sets $D_N$ and $\widehat{D_N}$. Then we will proceed to estimate $Mw_N$ on $D_N$ and $Tw_N$ on $\widehat{D_N}$, and we will complete the details of the construction of $w_N$ so that the conclusion is reached.

\bigskip

\textbf{The triadic decomposition.} For $k\in\mathbb Z$, we say that $Q$ is a triadic cube of the $k$-th generation in $\real^n$, if $Q$ has edge length $3^{-k}$ and its vertices are points of the grid $3^{-k} \mathbb Z^n$. For any cube $Q=Q(x,R)$ we define its triadic middle child as $\widehat{Q}=Q(x,R/3)$. For $k=0,1,2,\ldots$ we will consider $\mathcal T_k$ as a family of triadic cubes of the $(Nk)$-th generation, with $N\in\mathbb N$ fixed. We define these families inductively. We begin with $\mathcal T_0=\{[0,1]^d\}$. Once $\mathcal T_k$ is determined, for each $Q\in\mathcal T_k$ we will select a family $\mathcal T_{k+1}(Q)$ of triadic subcubes so that $\mathcal T_{k+1}(Q)\subset \{ \mbox{triadic } Q' \subset \widehat{Q}, \ |Q'|=3^{-Nd}|Q|\}$ and $\sharp T_{k+1}(Q)= A\sim 3^{(N-1)d}$, with $A\in\mathbb N$ a fixed number depending neither on $Q$ nor on $k$. The exact way of selecting these cubes will be explained later. Then we take $\mathcal T_{k+1}=\bigcup_{Q\in \mathcal T_k} \mathcal T_{k+1}(Q)$. 

\bigskip

Contained in each $Q\in\mathcal T_k$ we consider a triadic cube $J(Q)$ such that $|J(Q)|=|Q'|=3^{-Nd(k+1)}$ for any $Q'\in \mathcal T_{k+1}$. We will place $J(Q)$ having disjoint interior with respect to $\widehat{Q}$ but contiguous to it, in the sense that their boundaries intersect. In particular, the elements of the family $\displaystyle  \left\{J(Q)\right\}_{Q\in\bigcup_{k=0}^\infty  \mathcal T_k}$ are all disjoint. Moreover, if $N\ge 3$ and $Q_0\in \mathcal T_{k_0}$, for some $k_0$, 
\begin{equation}\label{distance}
\rm{dist}\left(J(Q_0),\left[\bigcup_{k=0}^\infty \bigcup_{Q\in \mathcal T_{k}}J(Q)\right]\setminus J(Q_0)\right)
\geq \frac{\ell}3-\frac{\ell}{3^N}\geq \frac{\ell}4,
\end{equation}
with $\ell=|J(Q)|^{1/d}$.

\bigskip

\textbf{The construction of the weight.} We define a weight $w_N$ supported in 
$$D_N = \bigcup_{k=0}^\infty \bigcup_{Q\in \mathcal T_k} J(Q),$$
 so that $w_N$ is constant over each $J(Q)$ and if $x\in J(Q)$ with $Q\in \mathcal T_{k}$ one has
\begin{equation}\label{cond.medida}
\alpha_{k} = w_N(x) = \frac{w_N(J(Q))}{|J(Q)|} = \frac{w_N(Q')}{|Q'|},
\end{equation}
for any $Q'\in \mathcal T_{k+1}$. In this way
\[
w_N(x)=\sum_{k=0}^\infty \alpha_k\sum_{Q\in \mathcal T_k} \chi_{J(Q)}.
\]
Observe that for $Q\in \mathcal T_k$
\[
w_N(Q)=w_N(J(Q))+w_N(\widehat{Q})=w_N(J(Q))+\sum_{Q'\in T_{k+1}(Q)} w_N(Q').
\]
Using (\ref{cond.medida}), the previous formula can be rewritten as 
\[
\alpha_{k-1} |Q| = \alpha_{k} |J(Q)|+\alpha_{k}\sharp\mathcal T_{k+1}(Q)\ |J(Q)| = \alpha_{k} |J(Q)|+\alpha_{k}A\ |J(Q)|,
\]
obtaining that
\[
\frac{\alpha_{k}}{\alpha_{k-1}}=\frac{3^{Nd}}{1+A}=:a.
\]
Hence, $\alpha_k=a^k\alpha_0$, for certain $\alpha_0$ and
\begin{eqnarray*}
w_N([0,1]^d)&=&\sum_{k=0}^\infty \sum_{Q\in\mathcal T_k} w_N(J(Q))= \alpha_0 \sum_{k=0}^\infty \sharp \mathcal T_k\  |J(Q)|\ a^k = \alpha_0\ \sum_{k=0}^\infty A^k 3^{-Nd(k+1)}a^k \\&=& \alpha_0\ 3^{-Nd}\,\sum_{k=0}^\infty \left( \frac{A}{1+A}\right)^k = \alpha_0\ 3^{-Nd}(1+A)=\alpha_0/a.
\end{eqnarray*}
We take $\alpha_0=a$ so that $w_N$ is a probability measure and $w_N\geq a>1$ in $D_N$, as stated.

\bigskip

\textbf{Controlling the maximal function.} We prove here that $Mw_N \leq Cw_N$ in $D_N$, with a constant $C$ independent of $N$. Fix $x\in J(Q)$ with $Q\in \mathcal T_k$ and take an arbitrary cube $R$ containing $x$. We want to show that 
$$\displaystyle \frac{w_N(R)}{|R|}\le Cw(x).$$
If  $|R|^{1/d} < 1/4|J(Q)|^{1/d}$, then $R\cap D_N=R\cap J(Q)$ from (\ref{distance}). This says that $w$ is constant in $R\cap J(Q)$ and the result is obvious. If, on the contrary, $|R|^{1/d}\geq 1/4|J(Q)|^{1/d}$ and we consider
\[
\mathcal A=\{ \mbox{triadic } Q',\ Q'\cap R\neq \emptyset,\ |Q'|=|J(Q)|\},
\]
then  $\sum_{Q'\in \mathcal A} |Q'|\leq 9^d |R|$. We claim that 
if $L\subset [0,1]^d$ is a triadic cube with size $|L|=|J(Q)|$ then
\begin{equation}\label{posibles.valores}
w_N(L) \leq \alpha_k |L|. 
\end{equation}
Using this, one has
\[
\frac{w_N(R)}{|R|}\leq \frac{1}{|R|}\sum_{Q'\in \mathcal A} w_N(Q')\leq  \frac{{\alpha_{k}}}{|R|}\sum_{Q'\in \mathcal A} |Q'| \leq 9^d\ w_N(x).
\]
The proof of (\ref{posibles.valores}) is easy. We have three possible situations: 

\begin{enumerate}[i)]
\item $L \cap D_N=\emptyset$, and there is nothing to show.

\item $L\subset J(Q_0)$, for some $Q_0 \in \mathcal T_j$ and $j\le k$. In this case $w_N$ is constant in $L$ with value $\alpha_j$. Since $\alpha_j\leq \alpha_k$, the result follows immediately. 

\item $L=Q'$ for some $Q' \in \mathcal T_{k+1}$. Here we have directly $w_N(L) =\alpha_k |L|$ by definition. 
\end{enumerate}

\bigskip

\textbf{Splitting $T w_N$ into `continuous' and `discrete' pieces.} Taking, by a slight abuse of notation, $\widehat{D_N}:=\bigcup_{k=0}^\infty \bigcup_{Q\in\mathcal T_{k}} \widehat{J(Q)}$, we want to prove that $|T w_N|\geq CN w_N$ in $\widehat{D_N}$. 

\bigskip

Let $x\in \widehat{J(Q)}$, with $Q\in \mathcal T_{k}$.  Then we have
\begin{eqnarray*}
T w_N(x) &=& \int_{Q^c}\! K(x,y)w_N(y)\,dy+ \int_{Q\setminus J(Q)}\!K(x,y)w_N(y)\,dy +\  p.v.\int_{J(Q)}\! K(x,y)w_N(y)\,dy\\
&=& I+ I\!I+ I\!I\!I.
\end{eqnarray*}
We further split $I$ and $I\!I$ into a `continuous' and a `discrete' part. Denoting by $c_R$ the center of a cube $R$, we have
\begin{eqnarray*}
I &=& \sum_{\begin{array}{c}\\[-6mm]\scriptstyle{L \, 	\rm{triadic} }\\[-1.5mm] \scriptstyle{|L|=|Q|, \, L\neq Q}\end{array}} \int_L K(x,y)w_N(y)\,dy 
\\&=& \sum_{\begin{array}{c}\\[-6mm]\scriptstyle{L \, 	\rm{triadic} }\\[-1.5mm] \scriptstyle{|L|=|Q|, \, L\neq Q}\end{array}} K(c_Q,c_L)w_N(L)\,dy+
\sum_{\begin{array}{c}\\[-6mm]\scriptstyle{L \, 	\rm{triadic} }\\[-1.5mm] \scriptstyle{|L|=|Q|, \, L\neq Q}\end{array}}\int_L \bigl(K(x,y)-K(c_Q,c_L)\bigr)w_N(y)\,dy\\&=&
I_1+I_2,
\end{eqnarray*}
and
\begin{eqnarray*}
I\!I &=& \sum_{L\in\mathcal T_{k+1}(Q)} \int_L K(x,y)w_N(y)\,dy 
\\&=& \sum_{L\in\mathcal T_{k+1}(Q)} K(c_{J(Q)},c_L)w_N(L)\,dy+\sum_{L\in\mathcal T_{k+1}(Q)} \int_L \bigl(K(x,y)-K(c_{J(Q)},c_L)\bigr)w_N(y)\,dy\\&=&
I\!I_1+I\!I_2.
\end{eqnarray*}

\bigskip

First, we will show that the `continuous' parts $I_2$, $I\!I_2$ and $I\!I\!I$ are `small' in the sense that $|I_2|+|I\!I_2|+|I\!I\!I|\lesssim w_N(x)$. Then we will show that $I\!I_1$ is much bigger than $w_N$ by showing that $|I\!I_1|\gtrsim N w_N(x)$. Although we will not have any control on $I_1$, we will  construct $J(Q)$ and ${\mathcal T_{k+1}}(Q)$ so that $I\!I_1$ has the same sign as $I_1$. In this way, we will have 
 $|I_1+I\!I_1|\geq |I\!I_1| \gtrsim N w_N(x)$. At that point we will get
\[
|T w_N(x)| \geq |I_1+I\!I_1| - |I_2+I\!I_2+I\!I\!I| \geq (cN-C)\ w_N(x)\geq CN w_N(x),
\]
for sufficiently large $N$. This would prove the result.

\bigskip

\textbf{The `continuous' pieces.} We recall the well-known fact (see \cite{Stein} for instance) that our hypotheses on $K$ imply the following estimates: there exist $\delta, \eta>0$ so that
\begin{equation}\label{CZy}
|K(x,y)-K(x,\bar y)|\leq C\frac{|y-\bar y|^\delta}{|x-y|^{d+\delta}},
\end{equation}
if $|x-y|>(1+\eta)|y-\bar y|$, and 
\begin{equation}\label{CZx}
|K(x,y)-K(\bar x,y)|\leq C\frac{|x-\bar x|^\delta}{|x-y|^{d+\delta}},
\end{equation}
if $|x-y|>(1+\eta)|x-\bar x|$. These estimates give rise to the so called $\delta$--Calder\'on-Zygmund kernels. Although in our case we have $\delta=1$, it is worth observing that this part of the construction works for these more general kernels too. 

\bigskip

When estimating $I_2$, first we use that $x\in \widehat{Q}$ and $y\in \widehat{L}$ to deduce
\[
|x-y|\sim |x-c_L|\sim |c_Q-c_L|,
\]
and as a consequence
\begin{eqnarray}\label{tamanyo.nucleo}
|K(x,y)-K(c_Q,c_L)|&\leq& |K(x,y)-K(x,c_L)|+|K(x,c_L)-K(c_Q,c_L)| \nonumber\\&\lesssim& \frac{|x-c_Q|^\delta}{|x-y|^{d+\delta}}+\frac{|y-c_L|^\delta}{|x-y|^{d+\delta}}\lesssim \frac{|Q|^{\delta/d}}{|y-x|^{d+\delta}}.
\end{eqnarray}
Hence,
\begin{eqnarray*}
|I_2|&\lesssim & |Q|^{\delta/d} \sum_{\begin{array}{c}\\[-6mm]\scriptstyle{|L|=|Q|}\\[-1.5mm] \scriptstyle{[0,1]^n\supset L\neq Q}\end{array}} \int_L \frac{w_N(y)}{|x-y|^{d+\delta}}\,dy\\&\leq& |Q|^{\delta/d} \int_{|x-y|>|Q|^{1/\delta}/4} \frac{w_N(y)}{|x-y|^{d+\delta}}\,dy \lesssim Mw_N(x).
\end{eqnarray*}
The last inequality follows from the fact that  $x\mapsto |x|^{-d-\delta}$ is a radially decreasing function and
\[
\int_{|x-y|>|Q|^{1/\delta}/4} \frac{1}{|x-y|^{d+\delta}}\,dy = \bigl|\mathbb S^{d-1}\bigr|_{d-1} \int_{|Q|^{1/d}/4}^\infty \frac{1}{t^{1+\delta}}\,dt \sim \frac{1}{|Q|^{\delta/d}}.
\]
(See \cite{Stein}.)
\bigskip

We estimate $I\!I_2$ in a  similar way. Since $J(Q)$ is not contained in $\widehat{Q}$, for $x\in J(Q)$, $y\in L \in {\mathcal T_{k+1}}(Q)$ and $v_{J(Q)}\in J(Q)$ to be determined later, one has
\[
|x-y|\sim |x-c_L|\sim |v_{J(Q)}-c_L|,
\]
and
\begin{eqnarray*}
|K(x,y)-K(v_{J(Q)},c_L)|&\leq& |K(x,y)-K(x,c_L)|+|K(x,c_L)-K(v_{J(Q)},c_L)| \\&\lesssim& \frac{|J(Q)|^{\delta/d}}{|y-x|^{d+\delta}}.
\end{eqnarray*}
Then, reasoning as before we obtain again
\begin{eqnarray*}
|I\!I_2|&\lesssim & |J(Q)|^{\delta/d} \sum_{L\in \mathcal T_{k+1}(Q)} \int_L \frac{w_N(y)}{|x-y|^{d+\delta}}\,dy\\&\leq& |J(Q)|^{\delta/d} \int_{|x-y|>|J(Q)|^{1/\delta}/3} \frac{w_N(y)}{|x-y|^{d+\delta}}\,dy \lesssim Mw_N(x).
\end{eqnarray*}

\bigskip

In order to bound $I\!I\!I$, we use that $w_N$ is constant over $J(Q)$ and the cancellation property of $K$ on $E=\{y:|x-y|<|\widehat{J(Q)}|^{1/d}\}$ to obtain
\begin{eqnarray}
I\!I\!I &=& w_N(x)\ p.v.\int_{J(Q)} \frac{\Omega(x-y)}{|x-y|^{d}}\,dy \nonumber\\&=& w_N(x)\ \int_{J(Q)\setminus E} \frac{\Omega(x-y)}{|x-y|^{d}}\,dy.\label{cancellation}
\end{eqnarray}
Hence
\begin{eqnarray*}
|I\!I\!I| &\leq & w_N(x)\ \int_{J(Q)\setminus E} \frac {\|\Omega\|_{L^\infty}} {|\widehat{J(Q)}|} \,dy \lesssim w_N(x).
\end{eqnarray*}

\bigskip

\begin{xrem} Observe that in all the above estimates we have not needed a precise description of the construction of the families ${\mathcal T_{k+1}}(Q)$ and the cubes $J(Q)$. The only information we have used so far is that each $Q' \in {\mathcal T_{k+1}}(Q)$ is a triadic subcube of $\widehat Q$ of size $3^{-Nd(k+1)}$ and that $J(Q)$ is of the same size and \lq touches' $\widehat Q$ from the outside.

Another important observation is that for $Q \in {\mathcal T_{k}}$, the term $I_1=I_1(Q)$ does not depend 
on the triadic cubes of the next generation. In particular, $I_1(Q)$ is independent of ${\mathcal T_{i}}$, for all $i>k$. This is consistent with the inductive process that we use in order to define our weights $w_N$.
\end{xrem}

\bigskip

\textbf{The `discrete' pieces in a simpler case: Riesz Transforms.} To get some intuition of the construction, we will first consider a concrete example. Assume that $T$ is a Riesz Transform, that is $T=R_j$ for some $j\in \{1,2,\dots,d\}$, where
\[
R_j f(x)= c_d\ p.v.\! \int_{\real^n} \frac{x_j-y_j}{|x-y|^{d+1}}f(y)\,dy,
\]
and $c_d$ is a normalizating constant depending on the dimension. In this case, given $Q\in \mathcal T_k$ we choose $\mathcal T_{k+1}(Q)$ to consist of all the triadic subcubes of $ \widehat{Q}$ of size $3^{-Nd}|Q|\}$. We take $J(Q)$ to be a triadic cube of size $3^{-N(k+1)d}$ contiguous to $\widehat{Q}$ so that their boundaries only share a point, hence a vertex. For $x\in \mathbb R^d$, we denote by $x_j$ its $j$-th coordinate. Now,
if $I_1\geq 0$ we place $J(Q)$ so that $\min_{x\in J(Q)}x_j\geq \max_{x\in \widehat{Q}} x_j$ and if $I_1\leq 0$ we require instead $\max_{x\in J(Q)}x_j\geq \min_{x\in \widehat{Q}} x_j$. This makes the signs of $I_1$ and $I\!I_1$ coincide. Calling 
$$\mathcal T_{k+1}^i(Q) = \{L\in \mathcal T_{k+1}(Q): |(c_L)_j-(c_{J(Q)})_j|=|c_L-c_{J(Q)}|_\infty= 
3^{-N(k+1)}i\},$$ 
and taking $v_{J(Q)}=c_{J(Q)}$ we have
\begin{eqnarray*}
|I\!I_1|&\gtrsim&\sum_{L\in\mathcal T_{k+1}(Q)} \frac{|(c_L)_j-(c_{J(Q)})_j|}{\bigl|c_L-c_{J(Q)}\bigr|_\infty^{d+1}}\  w_N(L) =  a^{k+1}|J(Q)| \sum_{i=1}^{3^{N-1}}\sum_{L\in\mathcal T_{k+1}^i(Q)} \frac{1}{\bigl|c_L-c_{J(Q)}\bigr|_\infty^d} 
\\&=& w_N(x)|J(Q)| \sum_{i=1}^{3^{N-1}} \frac{i^{d-1}}{(3^{-N(k+1)}i)^d} \,=\,w_N(x) \sum_{i=1}^{3^{N-1}} \frac 1i \gtrsim N w_N(x).
\end{eqnarray*}
Observe also that in this case $A=3^{(N-1)d}$ and, therefore, $\displaystyle a=\frac {3^{Nd}}{1+A}\sim 3^d$.

\bigskip

\textbf{Finishing the construction of the measure for a general operator.} We will now explain how we chose $J(Q)$ and $\mathcal T_{k+1}(Q)$ so that $I\!I_1$ behaves the way we need when $T$ is a general Calder\'on-Zygmund operator. This choice will depend on $T$.

\bigskip

Since $\Omega$ is a continuous function over the sphere with null integral mean, there exist $\lambda>0$, $r>0$ and two points in the sphere $z_+$ and $z_-$ so that for any $y\in B^+= B(z_+,r)\cap \mathbb S^{d-1}$ one has $\Omega(y)>\lambda$ and for any $y\in B^-=B(z_-,r)\cap \mathbb S^{d-1}$ one has $\Omega(y)<-\lambda$. We have the same bounds for $\Omega$ all over the cones $U^+=\{tx: t>0, x\in B^+\}$ and $U^-=\{tx: t>0, x\in B^-\}$. Using a rotation if necessary, we can assume that $z_+$ and $z_-$ are symmetric with respect to all the coordinate axis and that none of their coordinates are zero. This can be expressed in terms of coordinates with the relation $|(z_+)_i|=|(z_-)_i|\neq 0$ for all $i=1,\ldots,d$, or $z_+= \tau z_-$ with 
\[
\tau=\left(\delta_{i,j} \frac{\sign(z_-)_i}{\sign(z_+)_j}\right)_{i,j=1,\cdots,d}=\left[\begin{array}{cccc} \pm 1 & 0 & \cdots  & 0\\ 0&\pm 1&\cdots&0\\ \vdots&\vdots& \ddots& \vdots \\ 0&0&\cdots&\pm 1 \end{array}\right].
\]
Note that also $U^-=\tau U^+$.

\bigskip

For a $Q\in \mathcal T_{k}$ we denote by $v_+$ (respectively, $v_-$) the only vertex of $\widehat{Q}$ such that the half-line $s_+\equiv v_++t z_+$  (respectively, $s_-\equiv v_-+t z_-$), for $t>0$,  intersects the interior of $\widehat{Q}$. If $I_1\geq 0$ we will choose $v=v_{J(Q)}:=v_+$, $z=z_+$ and $U=U^+$. On the other hand, if $I_1\leq 0$ we choose $v=v_{J(Q)}:=v_-$, $z=z_-$ and $U=U^-$. Now we take $J(Q)$ to be the only triadic cube of size $3^{-Nd}|Q|$ so that the boundaries of $J(Q)$ and of $ \widehat{Q}$ intersect only at $v$. Once this is done we take 
\[
\mathcal T_{k+1}(Q)=\{\mbox{triadic } R\subset \widehat{Q} :\ |R|=3^{-Nd}|Q|,\ c_R \in v+U\}.
\]
The construction guarantees that $A=\sharp {\mathcal T}_{k+1}(Q)\sim 3^{(N-1)d}$ is independent of $k$ and $Q$, as required before.

\bigskip

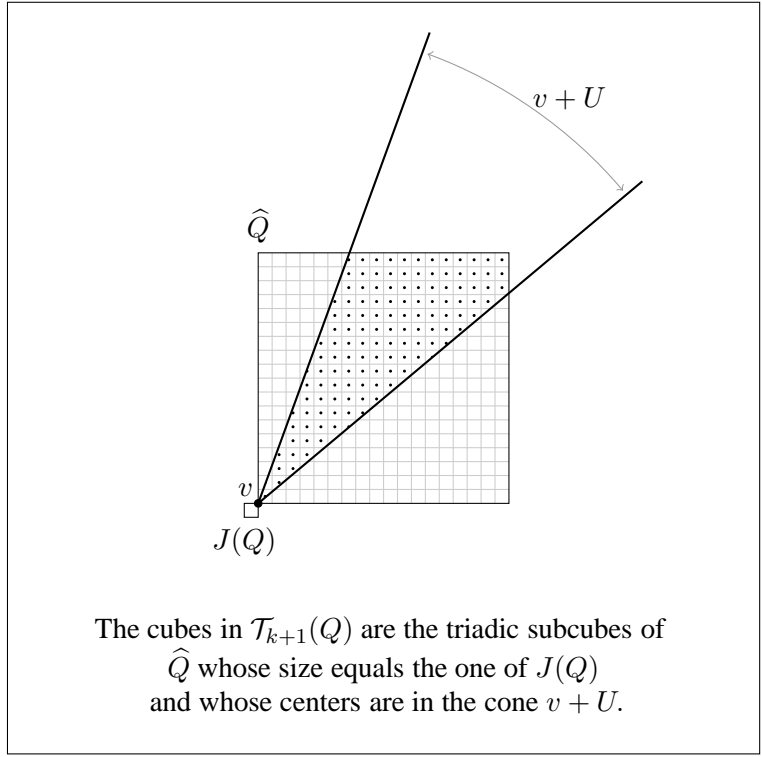
\begin{center}
\begin{tikzpicture}

\def\numero{18}
\def\nine{10}
\pgfmathsetmacro\three{\nine/3}
\pgfmathsetmacro\six{2*\three}
\pgfmathsetmacro\r{\three/\numero}

\coordinate (O) at (0,0);
\coordinate (A) at (\three,\three);
\coordinate (B) at (\three,\six);
\coordinate (C) at (\six,\six);
\coordinate (D) at (\six,\three);

\draw (O) -- (\nine,0) -- (\nine,\nine) -- (0,\nine) -- cycle;

\path [name path=lineal] (B) -- (C) -- (D);

\draw (A) -- ($(A)-(\r,0)$)--($(A)-(\r,\r)$)--($(A)-(0,\r)$)--cycle;
\draw (0,\nine) node [above] {$Q\in \mathcal T_k$};
\draw ($(A)-(\r,\r)$) node [below] {$J(Q)$};
\draw (\three,\six) node [above] {$\widehat{Q}$};
\draw [fill] (A) circle (0.5mm) node [left,yshift=5pt,xshift=2pt] {$v$};

\coordinate (H) at (A);
\coordinate (J) at (A);
\coordinate (K) at (D);
\coordinate (L) at (B);

\foreach \i in {2,...,\numero}
	{
	\coordinate (J) at ($(J)+(0,\r)$);
	\coordinate (K) at ($(K)+(0,\r)$);	
	\draw [color=black!20] (J) -- (K);
	\coordinate (L) at ($(L)+(\r,0)$);
	\coordinate (H) at ($(H)+(\r,0)$);
	\draw [color=black!20] (L) -- (H);
	}

\def\ana{70}
\def\anb{40}

\path [draw,thick,name path=lineau1] (A) -- ($(A)+(\ana:\six)$);
\path [draw,thick,name path=lineau2] (A) -- ($(A)+(\anb:\six)$);

\path [draw,color=black!40,<->,shorten >=2pt,shorten <=2pt] ($(A)+(\ana:11.5/12*\six)$) arc (\ana:\anb:11.5/12*\six);
\draw ($(A)+(\ana/2+\anb/2:\six)$) node [yshift=-2pt,xshift=9pt] {$v+U$};

\def\anglealfa{65}
\def\anglebeta{47}

\path [name intersections={of=lineal and lineau1}];
\coordinate (I) at (intersection-1);

\path [name intersections={of=lineal and lineau2}];
\coordinate (J) at (intersection-1);

\coordinate (segunda) at ($(A)+(\r/2,\r/2)$);

\begin{scope}
	\clip (A)--(J)--(C)--(I)--cycle;
	\foreach \i in {1,...,\numero}
		{
		\coordinate (primera) at (segunda);	
		\foreach \j in {1,...,\numero}
			{
			\fill (primera) circle (0.2mm);
			\coordinate (primera) at ($(primera)+(\r,0)$);
			}
		\coordinate (segunda) at ($(segunda)+(0,\r)$);			
		}
\end{scope}

\pgfmathsetmacro\numerito{\numero-3}

\pgfmathsetmacro\up{\three/2.5}%
\draw (\nine/2,0) node [above,yshift=3 mm] {$\begin{array}{c} \mbox{The cubes in }\mathcal T_{k+1}(Q)\mbox{ are the triadic subcubes of }\\ \widehat{Q}\mbox{ whose size}\mbox{ equals the one of }J(Q) \\\mbox{ and whose centers are in the cone }v+U.\end{array}$};

\draw (A) -- (B) -- (C) -- (D) -- cycle;

%\draw [->,black!50] ($(J)+(+\r/2,-\r/2)$)--($(D)+(3*\r,\numerito*\r-\r/2)$) node [right,black] {$\mathcal T_{k+1}^i(Q)$};

%\coordinate (uno) at ($(A)+(-5*\r,\r/2)$);
%\draw (uno) node [right] {$1$} ($(uno)+(0,\r)$) node [right] {$2$} ($(uno)+(0,\numerito*\r-\r/2)$) node [right] {$i$} ($(uno)+(0,\three-\r/2)$) node [right] {$3^{N-1}$};

\end{tikzpicture}
\end{center}

\bigskip

\textbf{Estimating the `discrete' pieces.} Since $c_L\in W$ for all $L\in \mathcal T_{k+1}(Q)$ we have
\[
|I\!I_1|=\left|\sum_{L\in\mathcal T_{k+1}(Q)} K(v,c_L)w_N(L)\right|\geq  \lambda a^{k+1} |J(Q)| \sum_{L\in\mathcal T_{k+1}(Q)} \frac{1}{|c_L-v|^d}.
\]

\bigskip

We want to find a lower estimate for the last sum. We could use an argument similar to the one for the Riesz transforms but we will use a more direct one. For a positive integer $i$ we define
\[
\Gamma_i=\{x\in v+U\cap \widehat{Q}: 3^{i-1}\,3^{-N(k+1)}<|x-v|\leq 3^i\ 3^{-N(k+1)}\}.
\]
We also define
\[
\mathcal T_{k+1}^i(Q)=\{ R\in \mathcal T_{k+1}(Q): c_R\in \Gamma_i\}.
\]

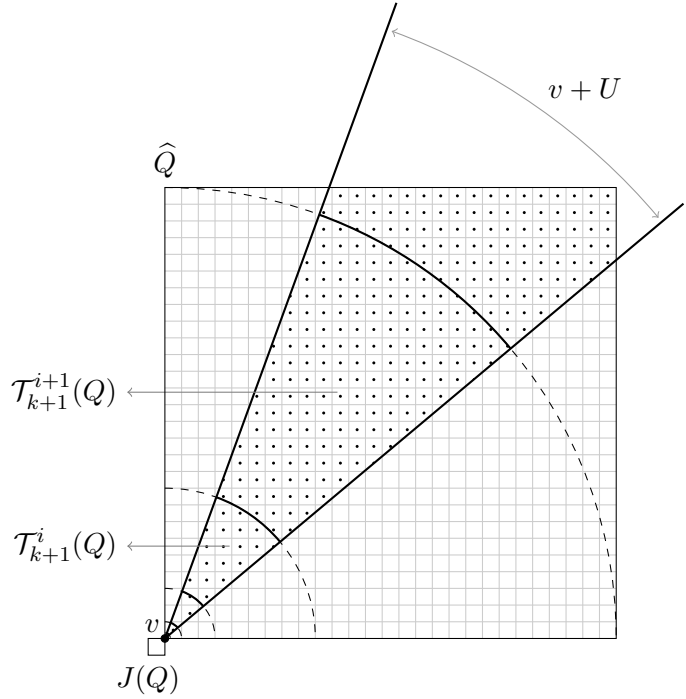
\bigskip

\begin{center}
\begin{tikzpicture}

\def\numero{27}
\def\nine{18}
\pgfmathsetmacro\three{\nine/3}
\pgfmathsetmacro\six{2*\three}
\pgfmathsetmacro\r{\three/\numero}
\pgfmathsetmacro\cono{\nine/2}

\coordinate (O) at (0,0);
\coordinate (A) at (\three,\three);
\coordinate (B) at (\three,\six);
\coordinate (C) at (\six,\six);
\coordinate (D) at (\six,\three);

\path [name path=lineal] (B) -- (C) -- (D);

\draw (A) -- ($(A)-(\r,0)$)--($(A)-(\r,\r)$)--($(A)-(0,\r)$)--cycle;
\draw ($(A)-(\r,\r)$) node [below] {$J(Q)$};
\draw (\three,\six) node [above] {$\widehat{Q}$};
\draw [fill] (A) circle (0.5mm) node [left,yshift=5pt,xshift=2pt] {$v$};

\coordinate (H) at (A);
\coordinate (J) at (A);
\coordinate (K) at (D);
\coordinate (L) at (B);

\foreach \i in {2,...,\numero}
	{
	\coordinate (J) at ($(J)+(0,\r)$);
	\coordinate (K) at ($(K)+(0,\r)$);	
	\draw [color=black!20] (J) -- (K);
	\coordinate (L) at ($(L)+(\r,0)$);
	\coordinate (H) at ($(H)+(\r,0)$);
	\draw [color=black!20] (L) -- (H);
	}

\def\ana{70}
\def\anb{40}

\path [draw,thick,name path=lineau1] (A) -- ($(A)+(\ana:\cono)$);
\path [draw,thick,name path=lineau2] (A) -- ($(A)+(\anb:\cono)$);

\path [draw,color=black!40,<->,shorten >=2pt,shorten <=2pt] ($(A)+(\ana:11.5/12*\cono)$) arc (\ana:\anb:11.5/12*\cono);
\draw ($(A)+(\ana/2+\anb/2:\cono)$) node [yshift=-2pt,xshift=12pt] {$v+U$};

\def\anglealfa{65}
\def\anglebeta{47}

\path [name intersections={of=lineal and lineau1}];
\coordinate (I) at (intersection-1);

\path [name intersections={of=lineal and lineau2}];
\coordinate (J) at (intersection-1);

\coordinate (segunda) at ($(A)+(\r/2,\r/2)$);

\begin{scope}
	\clip (A)--(J)--(C)--(I)--cycle;
	\foreach \i in {1,...,\numero}
		{
		\coordinate (primera) at (segunda);	
		\foreach \j in {1,...,\numero}
			{
			\fill (primera) circle (0.2mm);
			\coordinate (primera) at ($(primera)+(\r,0)$);
			}
		\coordinate (segunda) at ($(segunda)+(0,\r)$);			
		}
\end{scope}

\foreach \i in {1,3,9,27,81}
		{
		\coordinate (I) at ($(A)+(90:\three/\i)$);
		\draw [dashed] (I) arc (90:0:\three/\i);
		\coordinate (J) at ($(A)+(\ana:\three/\i)$);
		\draw [thick] (J) arc (\ana:\anb:\three/\i);		
		}
\pgfmathsetmacro\seno{sin(\ana/2+\anb/2)}

\draw [->,black!50] ($(A)+(\ana/2+\anb/2:\three/1.5)$)--($(A)+(-0.5,\seno*\three/1.5)$) node [left,black] {$\mathcal T_{k+1}^{i+1}(Q)$};

\draw [->,black!50] ($(A)+(\ana/2+\anb/2:\three/4)$)--($(A)+(-0.5,\seno*\three/4)$) node [left,black] {$\mathcal T_{k+1}^i(Q)$};

\draw (A) -- (B) -- (C) -- (D) -- cycle;

\end{tikzpicture}
\end{center}

\bigskip

Now we choose $N$ large enough to make $J(Q)$ very small compared to $\Gamma_{\lfloor N/2\rfloor}$, so that the measure of $\Gamma_{i}$ is comparable to the sum of the measures of the cubes in $\mathcal T_{k+1}^i$ for $\lfloor N/2\rfloor\leq i\leq N-1$, that is
\[
|\Gamma_{i}|\ \ \sim \sum_{R\in \mathcal T_{k+1}^{i}(Q)} |R|\ \ =\ \ \sharp \mathcal T_{k+1}^{i}(Q)\ |J(Q)|.
\]
Note that $|J(Q)|=3^{-Nd(k+1)}$ and $|\Gamma_i|= \beta\  (3^{d(i-N(k+1))}-3^{d(i-1-N(k+1))})=\ 2\beta\ 3^{d(i-1-N(k+1))}$ for certain $\beta>0$ that depends on the opening of the cube $U$. The previous choice of $N$ is possible since the quotient of the measures of $\Gamma_{\lfloor N/2\rfloor}$ and $J(Q)$ is of the order of $3^{dN/2}$. The conclusion is that for $\lfloor N/2\rfloor\leq i\leq N-1$ one has $\sharp \mathcal T_{k+1}^i(Q) \sim 3^{di}$ and as a consequence
\begin{eqnarray*}
|I\!I_1| &\geq& \lambda w_N(x)\ |J(Q)| \sum_{i=\lfloor N/2\rfloor}^{N-1}\sum_{L\in \mathcal T_{k+1}^i(Q)} \frac1{|c_L-v|^d} \\&\gtrsim& \lambda w_N(x)|J(Q)| \sum_{i=\lfloor N/2\rfloor}^{N-1} \frac{\sharp \mathcal T_{k+1}^i(Q)}{|3^{i}\ 3^{-N(k+1)}|^d}\ \ \gtrsim\ \ \lambda N w_N(x).
\end{eqnarray*}
This finishes the proof of Proposition \ref{prop.pesos}.

\end{proof}

%%%%%%%%%%%%%%%%%%%%%%%%%%%%%%%%%%%
%%%%%%%%%%%%%%%%%%%%%%%%%%%%%%%%%%%
%%%%%%%%%%%%%%%%%%%%%%%%%%%%%%%%%%%
%%%%%%%%%%%%%%%%%%%%%%%%%%%%%%%%%%%

%%%%%%%%%%%%%%%%%%%%%%%%%%%%%%%%%%%
%%%%%%%%%%%%%%%%%%%%%%%%%%%%%%%%%%%
%%%%%%%%%%%%%%%%%%%%%%%%%%%%%%%%%%%
%%%%%%%%%%%%%%%%%%%%%%%%%%%%%%%%%%%

\bigskip

\section{Final remarks.} 

\noindent\textbf{Variable Kernels.} We point out that most of the arguments of the previous proof also work if $K$ is a variable Calder\'on-Zygmund kernel with the standard size conditions. Thus, a similar construction is possible for such kernels, if in addition they have an adequate distribution of signs so that one can find cones defining $\mathcal T_{k+1}(Q)$ as before.

\bigskip

\noindent\textbf{Counterexamples for condition (\ref{cond.perez.cu}).} It is implicit in the proof of Theorem \ref{no.conj1} that the weights $w_N$ together with the functions $f_N=w_N T^\ast w_N/(Mw_N)^2$ give counterexamples for
the condition (\ref{cond.perez.cu}) established by C. P\'erez and D. Cruz-Uribe. As already pointed out in \cite{RegueraScurry} the election of $u$ and $v$ in the proof of Theorem \ref{no.conj2} gives again counterexamples for (\ref{cond.perez.cu}). The point in the given proof of Theorem \ref{no.conj1} is to produce an explicit counterexample for Conjecture \ref{conj1}. An interesting observation is that the weights $w_N$ do satisfy Conjecture \ref{conj1}. To see this, recall that (\ref{conjetura}) is true for $M$ replaced by $M^2$ and then apply the `local' $A_1$ condition $Mw_N\lesssim w_N$ in $D_N$.

\bigskip

\noindent\textbf{`Local $A_p$' weights.} It is clear that \lq local' $A_p$ weights share some of the properties of the usual Muckenhoupt $A_p$ weights. For example, it is easy to see  that conditions (\ref{a1.un.peso}) and (\ref{ap.un.peso}), satisfied on the support of the weight, are equivalent to the weak boundedness of $M$ on weighted $L^p$. However, there are some important differences too. One of them is the non existence of a reverse H\"older inequality for local weights. In fact, we have the following

\bigskip

\begin{lemma}
Let $w$ be the local $A_1$ weight defined in the proofs of Theorems \ref{no.conj2} and \ref{no.conj3}. Then, for all $\eps>0$, $w^{1+\eps}$ is not even a locally integrable function.
\end{lemma}

\bigskip

\begin{proof}
Observe that for each $N$
\begin{eqnarray*}
\int_{Q_N} w^{1+\eps}&=&\int_{D_N} w_N^{1+\eps}=\sum_{k=0}^\infty\left(\frac{3^{Nd}}{1+A}\right)^{(k+1)(1+\eps)}A^k\, 3^{-Nd(k+1)}
\\ &=& \frac 1{A} \sum_{k=0}^\infty\left(\frac{3^{\eps Nd}A}{(1+A)^{1+\eps}}\right)^{k+1}.
\end{eqnarray*}
Since $A\leq 3^{(N-1)d}$, by taking $N$ large so that $\displaystyle \frac{3^{\eps Nd}A}{(1+A)^{1+\eps}}>1$, we see that the series diverges to $\infty$. This shows that 
$w^{1+\eps}\notin L^1_{\rm loc}(\mathbb R^d).$
\end{proof}

\end{document}